\documentclass{amsart}

\usepackage{amsmath,amsxtra,amsthm,amssymb,amsfonts,graphicx} % amsrefs

\usepackage[pdfusetitle,bookmarks=true,bookmarksnumbered=false,bookmarksopen=false,breaklinks=false,pdfborder={0 0 0},backref=false,colorlinks=true]{hyperref}

\theoremstyle{plain}

\newtheorem{thm}{Theorem}[section]

\newtheorem{lem}[thm]{Lemma}

\newtheorem{alg}[thm]{Algorithm}

\theoremstyle{definition}

\newtheorem{exa}[thm]{Example}

\def\l{\left}
\def\r{\right}

\def\as{\!\mathrel{\mathop:}=}

\def\geo{\col[0,1]\to\hs}
\def\ran{\col\Omega\to S}

\def\nat{\mathbb{N}}
\def\nato{{\mathbb{N}_0}}
\def\rls{\mathbb{R}}
\def\exrls{(-\infty,\infty]}
\def\eps{\varepsilon}
\def\lam{\lambda}
\def\sph{\mathbb{S}}

\def\half{\frac{1}{2}}

\renewcommand{\phi}{\varphi}

\def\col{\colon}

\def\argmin{\operatornamewithlimits{\arg\min}}

\def\di{\operatorname{d}}

\def\expe{\mathbb{E}}                         % expectation
\def\prob{\mathbb{P}}

\def\mi{\operatorname{Min}}

\def\cf{\mathcal{F}}

\def\ts{\mathcal{T}}
\def\hs{\mathcal{H}}   

\begin{document}
\title[Stochastic minimization]{A variational approach to stochastic minimization of convex functionals}
\author[M. Ba\v{c}\'ak]{Miroslav Ba\v{c}\'ak}
\date{\today}
\thanks{}

\address{Max Planck Institute for Mathematics in the Sciences, Inselstr. 22, 04 103 Leipzig, Germany}
\email{bacak@mis.mpg.de}

\begin{abstract}
Stochastic methods for minimizing a convex integral functional, as initiated by Robbins and Monro in the early 1950s, rely on the evaluation of a gradient (or subgradient if the function is not smooth) and moving in the corresponding direction. In contrast, we use a variational technique resulting in an implicit stochastic minimization method, which has recently appeared in several diverse contexts. Such an approach is desirable whenever the underlying space does not have a differentiable structure and moreover it exhibits better stability properties which makes it preferable even in linear spaces. Our results are formulated in locally compact Hadamard spaces, but they are new even in Euclidean space, the main novelty being more general growth conditions on the functional. We verify that the assumptions of our convergence theorem are satisfied in a few classical minimization problems.
\end{abstract}

\maketitle
%%%%%%%%%%%%%%%%%%%%%%%%%%%%%%%%%%%%%%%%%%%%%%%%%%%%%%%%%%%%%%%%%%%%%%%%%%%%%%%%%%%%%%%%%%%%%%%%%%%%%%%%%%%%%%%%%%%%%%%%%%%%%%%%%%%%%%%%%%%%%%%%%%%%%%%%%%%%%%%%%%%%%%%%%%%%%%%%%%%%%%%%%%%%%%%%%%%%%%%%%%%%%%%%%%
%%%%%%%%%%%%%%%%%%%%%%%%%%%%%%%%%%%%%%%%%%%%%%%%%%%%%%%%%%%%%%%%%%%%%%%%%%%%%%%%%%%%%%%%%%%%%%%%%%%%%%%%%%%%%%%%%%%%%%%%%%%%%%%%%%%%%%%%%%%%%%%%%%%%%%%%%%%%%%%%%%%%%%%%%%%%%%%%%%%%%%%%%%%%%%%%%%%%%%%%%%%%%%%%%%
\section{Introduction}

In the present paper we study convex integral functionals on a locally compact Hadamard space $(\hs,d).$ The reader who is not interested in this level of generality may consider the underlying space to be Euclidean. Let $(S,\mu)$ be a probability space and assume that a function $f\col\hs\times S\to\exrls$ satisfies 
\begin{enumerate}
 \item $f(\cdot,\xi)$ is a convex lsc function for each $\xi\in S,$
 \item $f(x,\cdot)$ is a measurable function for each $x\in\hs.$
\end{enumerate}
Then define
\begin{equation} \label{eq:form}
 F(x)\as\int_S f\l(x,\xi\r)\di\mu(\xi),\qquad x\in\hs.
\end{equation}
We will assume that $F(x)>-\infty$ for every $x\in\hs$ and that $F$ is lsc (which can be assured, for instance, by Fatou's lemma). 

The aim of stochastic minimization is to minimize the function $F$ by using the marginal functions $f(\cdot,\xi).$ Stochastic (sub)gradient methods originated in the seminal work of H.~Robbins and S.~Monro~\cite{robbins-monro} and keep attracting significant attention \cite{kushner-yin}. In those algorithms, each approximation step consists of choosing randomly $\xi\in S$ and moving in the direction of a (sub)gradient of the function $f(\cdot,\xi).$ In contrast, the stochastic proximal point algorithm (PPA) appeared only very recently in \cite{wang-bertsekas} and the aim of our paper is to further extend this approach. Recall that the basic proximal point algorithm is due to Martinet~\cite{martinet}, Rockafellar~\cite{rocka} and Br\'ezis\&Lions \cite{brezis}.

The stochastic PPA is based on resolvent mappings associated with the marginal functions
\begin{equation*} 
J_\lambda^\xi x \as\argmin_{y\in \hs}\l[f(y,\xi)+\frac{1}{2\lambda}d(x,y)^2\r],\qquad x\in\hs,
\end{equation*}
where $\lam>0$ is a given parameter. Note that in Hilbert spaces one has $J_\lambda^\xi =\l(I+\lam\partial f(\cdot,\xi)\r)^{-1}.$ In Hadamard spaces, the resolvent of a convex function was first studied by J.~Jost \cite{jost95} and U. Mayer~\cite{mayer}. The \emph{stochastic proximal point algorithm} is defined as follows.
\begin{alg}
 Let $\l(\lam_i\r)$ be a sequence of positive reals and let $\l(\xi_i\r)$ be a sequence of independent random variables $\xi_i\ran$ with distribution $\mu.$ Choose an arbitrary starting point $x_0\in \hs$ and set
\begin{equation}\label{eq:sppa}
 x_i\as J_{\lam_i}^{\xi_i} x_{i-1},
\end{equation}
for each $i\in\nat.$ Note that $x_i\ran$ is a random variable.
\end{alg}
We prove the convergence of the sequence $\l(x_i\r)$ to a minimizer of $F$ under weaker assumptions than in the existing literature and moreover we demonstrate the applicability of the stochastic PPA into several classical optimization problems in Euclidean space as well as a recent statistical model for phylogenetic inference \cite{benner}, where the underlying space is Hadamard. In all these examples, the resolvents of marginal functions are easy to compute in a closed form, which makes the algorithm readily usable in practice. The PPA as an implicit minimization method has better stability and convergence properties than gradient descent methods (see \cite{bertsekas} for detailed and authoritative arguments), which makes our algorithm attractive even in Euclidean space. Moreover, in Hadamard spaces without a differentiable structure, gradient descent methods cannot be used and the stochastic PPA is therefore the only possible option. Our choice of a Hadamard space setting in the present paper is, inter alia, justified by the minimization problem in Example \ref{exa:bhv} below. The relevance of Hadamard spaces for analysis and optimization is described in \cite{mybook}. In particular, minimization algorithms are discrete time analogs of gradient flows, which  in Hadamard spaces originated independently in \cite{jost-ch} and \cite{mayer} and were later generalized into much more general singular metric spaces \cite{ags}.

In the remainder of the Introduction, we describe the relation of our work to the existing literature. D.~Bertsekas \cite{bertsekas} was the first to introduce the proximal point algorithm for convex optimization problems with objective functions of the form
\begin{equation} \label{eq:sumfun} 
f\as\sum_{n=1}^N f_n,
\end{equation}
where $f_n\col\rls^d\to\rls$ are all convex and continuous. This is a special case of \eqref{eq:form}. Bertsekas' work \cite{bertsekas} has been very influential: the algorithm was generalized into locally compact Hadamard space \cite{mm} and subsequently applied in computational phylogenetics \cite{benner}. S.~Ohta and M.~P\'alfia \cite{ohta-palfia} then continued along these lines and proved convergence in some positively curved metric spaces (also considering marginal functions over general probability space). The algorithm has also been extended into the $1$-dimensional sphere~$\sph^1$ and applied in image restoration \cite{bergmann-etal,bw1,bw2}; with further developments following in~\cite{imaging}. For a survey, see G.~Steidl's paper \cite{steidl}. Another interesting result on the PPA in Hadamard spaces is due to S.~Banert \cite{banert}.

M. Wang and D. Bertsekas \cite{wang-bertsekas} also minimize a convex function of the form \eqref{eq:form} by the stochastic proximal point algorithm, but their goals and means are somewhat different from ours. Namely, we work in locally compact Hadamard spaces  whereas \cite{wang-bertsekas} requires $\rls^d,$ and moreover, our growth conditions on the marginal functions $f(\cdot,\xi)$ are weaker and allow applications like \eqref{eq:leastsquares}, whereas \cite{wang-bertsekas} requires the existence of a \emph{common} Lipschitz constant $L$ in \cite[Assumption 1]{wang-bertsekas} and a Lipschitz-like condition with constant $L$ on all subgradients, which excludes important applications such as Example \ref{exa:e3}. This makes our results new even in Euclidean spaces. 

An implicit stochastic minimization has also recently emerged in statistics \cite{toulis2,toulis1}. 

In conclusion, the present paper can be viewed as an continuation of and improvement upon \cite{mm} for allowing general probability distributions as well as weaker growth conditions on the marginal functions. Function~\eqref{eq:form} and~\eqref{eq:sumfun} are sometimes called \emph{expected risk} and \emph{empirical risk,} respectively. Naturally, we want to minimize expected risk rather than empirical risk provided the former is available (for instance, if we can sample arbitrary amount of data from $\mu$). Hence the present paper enables an expected risk minimization, whereas the results of~\cite{mm} applied to empirical risk only.

Our convergence theorem as well as its proof is ispired by Lyapunov's stability theory of dynamical systems. Specifically, the distance function squared (which is strongly convex in Hadamard spaces) plays the role of a Lyapunov function and we need to arrive at estimate~\eqref{eq:superm} in order to apply the Robbins-Siegmund lemma. Futhermore, our Lipschitz-like growth condition~\eqref{i:sppa:lips} from Theorem \ref{thm:conv} is also familiar from the theory of dynamical systems. 

%%%%%%%%%%%%%%%%%%%%%%%%%%%%%%%%%%%%%%%%%%%%%%%%%%%%%%%%%%%%%%%%%%%%%%%%%%%%%%%%%%%%%%%%%%%%%%%%%%%%%%%%%%%%%%%%%%%%%%%%%%%%%%%%%%%%%%%%%%%%%%%%%%%%%%%%%%%%%%%%%%%%%%%%%%%%%%%%%%%%%%%%%%%%%%%%%%%%%%%%%%%%%%%%%%
\section{Preliminaries}

For a background in Hadamard space theory, the reader is referred to \cite{mybook,bh,jost2}. Here we recall that a complete metric space $(\hs,d)$ is called a \emph{Hadamard space} if for each pair of points $x,y\in \hs$ there exists a mapping $\gamma\geo,$ called a geodesic, such that $\gamma(0)=x,\gamma(1)=y,$ and
\begin{equation*}d\l(\gamma(s),\gamma(t)\r)=d(x,y)\:|s-t|,\end{equation*}
for each $s,t\in[0,1],$ and if for each point $z\in\hs,$ geodesic $\gamma\geo,$ and $t\in[0,1],$ we have
\begin{equation} \label{eq:cat}
d\l(z,\gamma(t)\r)^2\leq (1-t) d\l(z,\gamma(0)\r)^2+td\l(z,\gamma(1)\r)^2-t(1-t) d\l(\gamma(0),\gamma(1)\r)^2.
\end{equation}
We say that a function $f\col\hs\to\exrls$ is \emph{convex} if $f\circ\gamma\col[0,1]\to\exrls$ is a convex function. 

We shall need the following well known lemma.
\begin{lem} \label{lem:estim} If $h\col\hs\to\exrls$ is a convex lsc function and $J_\lam^h$ stands for its resolvent, then
\begin{equation*} h\l(J_\lam^h (x)\r)-h(y)\leq \frac1{2\lam} d(x,y)^2 -\frac1{2\lam} d\l(J_\lam^h (x),y\r)^2, \end{equation*}
for every $x,y\in \hs.$ 
\end{lem}
\begin{proof}
 See \cite[Lemma 2.2.23]{mybook}.
\end{proof}

Like in the above listed literature \cite{mm,bertsekas,wang-bertsekas}, our proof of Theorem \ref{thm:conv} uses the famous Robbins-Siegmund theorem~\cite{robbins-siegmund}.
\begin{thm} \label{thm:rs}
Let $\l(\Omega,\cf,\l(\cf_k\r)_{k\in\nato},\prob\r)$ be a filtered probability space. Assume $\l(U_k\r),\l(Y_k\r),\l(Z_k\r)$ and $\l(W_k\r)$ are sequences of nonnegative real-valued random variables defined on $\Omega$ and assume that
\begin{enumerate}
 \item $U_k,Y_k,Z_k,W_k$ are~$\cf_k$-measurable for each $k\in\nato,$
 \item $ \expe \l(Y_{k+1}\big|\cf_k\r) \leq \l(1+U_k\r)Y_k-Z_k+W_k,$ for each $k\in\nato,$ \label{i:mar:ii}
 \item $\sum_k U_k<\infty$ and $\sum_k W_k<\infty.$
\end{enumerate}
Then the sequence $\l(Y_k\r)$ converges to a finite random variable $Y$ almost surely, and $\sum_k Z_k<\infty$ almost surely.
\end{thm}
\begin{proof}
 \cite[Theorem~1]{robbins-siegmund}.
\end{proof}

%%%%%%%%%%%%%%%%%%%%%%%%%%%%%%%%%%%%%%%%%%%%%%%%%%%%%%%%%%%%%%%%%%%%%%%%%%%%%%%%%%%%%%%%%%%%%%%%%%%%%%%%%%%%%%%%%%%%%%%%%%%%%%%%%%%%%%%%%%%%%%%%%%%%%%%%%%%%%%%%%%%%%%%%%%%%%%%%%%%%%%%%%%%%%%%%%%%%%%%%%%%%%%%%%%
\section{Main convergence theorem}

In this Section we prove that if the marginal functions $f(\cdot,\xi)$ satisfy a rather mild growth condition, the stochastic PPA converges to a minimizer of the convex integral functional in question. Notice that the growth condition in~\eqref{i:sppa:lips} is somewhat delicate, because there is no absolute value on the left hand side. It is also weaker than \cite[Assumption 1]{wang-bertsekas}, which requires the existence of a constant $L>0$ such that for almost every $\xi\in S,$ the subgradients of $f(\cdot,\xi)$ satisfy a Lipschitz-like condition with Lipschitz constant~$L.$

\begin{thm}[Convergence of the stochastic PPA] \label{thm:conv}
Assume that
\begin{enumerate}
 \item a function $F$ is of the form \eqref{eq:form} and has a minimizer,
 \item there exists $p\in\hs$ and an $L^2$-function $L\col S\to(0,\infty)$ such that \label{i:sppa:lips}
\begin{equation*} 
 f\l(x,\xi\r)-f\l(y,\xi\r)\leq L(\xi)\l[1+d(x,p) \r]d(x,y),
\end{equation*}
whenever $x,y\in\hs,$
\item $\sum_{i=1}^\infty \lam_i =\infty$ and $\sum_{i=1}^\infty \lam_i^2 <\infty.$
\end{enumerate}
Then there exists a random variable $x\col\Omega\to\mi F$ such that for almost every $\omega\in\Omega$ the sequence $\l(x_i(\omega)\r)$ given by \eqref{eq:sppa} converges $x(\omega).$
\end{thm}
\begin{proof}

Denote $\cf_i\as\sigma\l(\xi_0,\dots,\xi_i\r).$ We claim that for each $y\in\hs$ there exits a constant $C_{y,p}>0,$ depending also on $p,$ such that 
\begin{equation} \label{eq:superm} \expe\l[ d\l(x_{i+1},y\r)^2\big| \cf_i\r] \leq \l(1+2C_{y,p}\lam_i^2 \expe \l[L\l(\xi\r)^2\r]\r) d\l(x_i,y\r)^2 -2\lam_i \l[F\l(x_i\r)-F(y)\r]+2C_{y,p}\lam_i^2 \expe \l[L\l(\xi\r)^2\r],
\end{equation}
almost surely and for every $i\in\nat.$ Here of course the sequence $x_i$ depends on $\omega\in\Omega$ and one should write $x_i(\omega)$ instead of $x_i$ to be more precise. We will now prove this claim.

Let us fix $i\in\nat.$ By Lemma \ref{lem:estim} we have
\begin{equation*} d\l(x_{i+1},y\r)^2\leq d\l(x_i,y\r)^2 - 2\lam_i\l[ f\l(x_{i+1},\xi_i\r)-f\l(y,\xi_i\r)\r].\end{equation*}
Taking the conditional expectation with respect to $\cf_i$ gives
\begin{equation*} \expe\l[ d\l(x_{i+1},y\r)^2\big| \cf_i\r] \leq d\l(x_i,y\r)^2 - 2\lam_i\expe\l[ f\l(x_{i+1},\xi_i\r)-f\l(y,\xi_i\r)\big| \cf_i\r].\end{equation*}
If we denote by $x_i^\xi$ the result of the algorithm at the $i$-th step if $\xi_i(\omega)=\xi$ we get
\begin{align*}
\expe\l[ d\l(x_{i+1},y\r)^2\big| \cf_i\r] & \leq d\l(x_i,y\r)^2 - 2\lam_i \expe \l[ f\l(x_{i+1}^\xi,\xi\r)-f\l(y,\xi\r)\r] \\ & = d\l(x_i,y\r)^2 - 2\lam_i \l[ F\l(x_i\r)-F(y)\r] + 2\lam_i \expe \l[ f\l(x_i,\xi\r)-f\l(x_{i+1}^\xi,\xi\r)\r] .
\end{align*}
By the assumptions we have
\begin{align*}
\expe \l[ f\l(x_i,\xi\r)-f\l(x_{i+1}^\xi,\xi\r)\r] &\leq  \l(1+d\l(x_i,p \r)\r)\expe \l[L\l(\xi\r)  d\l( x_i,x_{i+1}^\xi\r)\r] \\ &\leq 2\lam_i \l(1+d\l(x_i,p \r)\r)^2\expe \l[L\l(\xi\r)^2\r] \\ &\leq 4\lam_i \l(1+d\l(x_i,p \r)^2\r)\expe \l[L\l(\xi\r)^2\r], 
\end{align*}
since
\begin{equation*} d\l(x_i,x_{i+1}^\xi\r)\leq 2\lam_i \frac{f\l(x_i,\xi\r)-f\l(x_{i+1}^\xi,\xi\r)}{d\l( x_i,x_{i+1}^\xi\r)}\leq 2\lam_i L\l(\xi\r)\l(1+d\l(x_i,p \r)\r).\end{equation*}
Thus, for each $y\in\hs,$ there exits a constant $C_{y,p}>0$ such that
\begin{equation*}
 \expe \l[ f\l(x_i,\xi\r)-f\l(x_{i+1}^\xi,\xi\r)\r] \leq  C_{y,p}\lam_i \l(1+d\l(x_i,y \r)^2\r)\expe \l[L\l(\xi\r)^2\r], 
\end{equation*}
We hence finally obtain
\begin{equation*}
\expe\l[ d\l(x_{i+1},y\r)^2\big| \cf_i\r] \leq \l(1+2C_{y,p}\lam_i^2 \expe \l[L\l(\xi\r)^2\r]\r) d\l(x_i,y\r)^2 -2\lam_i \l[F\l(x_i\r)-F(y)\r]+2C_{y,p}\lam_i^2 \expe \l[L\l(\xi\r)^2\r],\end{equation*}
which finishes the proof of~\eqref{eq:superm}.

Next, choose a countable dense subset $\l\{v_n\r\}$ of $\mi F.$ This is possible because $\mi F$ is a locally compact Hadamard space, its closed balls are therefore compact by the Hopf-Rinow theorem \cite[p.~35]{bh} and consequently it is separable. Fix $v_n$ for a moment and for each $i\in\nat$ apply~\eqref{eq:superm} with $y=v_n$ to obtain
\begin{align*} \expe\l[ d\l(x_{i+1}(\omega),v_n\r)^2\big| \cf_i\r] & \leq \l(1+2C_{v_n,p}\lam_i^2 \expe \l[L\l(\xi\r)^2\r]\r) d\l(x_i(\omega),v_n\r)^2-2\lam_i \l[F\l(x_i(\omega)\r)-F\l(v_n\r)\r] \\  & \quad +2C_{v_n,p}\lam_i^2 \expe \l[L\l(\xi\r)^2\r],
\end{align*}
for every $\omega$ from a full measure set $\Omega_{v_n}\subset\Omega.$ Theorem \ref{thm:rs} immediately gives that $d\l(v_n,x_i(\omega)\r)$ converges and that
\begin{equation} \label{eq:fromrs} \sum_{i=0}^\infty \lam_i\l[F\l(x_i(\omega)\r)-\inf F\r]<\infty,\end{equation}
for every $\omega\in\Omega_{v_n}.$ Next denote
\begin{equation*}\Omega_\infty\as \bigcap_{n\in\nat} \Omega_{v_n},\end{equation*}
which is by countable subadditivity again a set of full measure. For each $\omega\in\Omega_\infty,$ we have from~\eqref{eq:fromrs} that
\begin{equation} \label{eq:liminf}
\liminf_{i\to\infty} F\l(x_i(\omega)\r)= \inf F.
\end{equation}
Since $\l(x_i(\omega)\r)$ is bounded, it has a cluster point $x(\omega)\in \hs.$ By the lower semicontinuity of~$F$ and by~\eqref{eq:liminf} we obtain that $x(\omega)\in\mi F.$

We will now show that, given $z\in\mi F,$ the sequence $d\l(z,x_i(\omega)\r)$ converges. Indeed, for each $\eps>0$ there exists $v_{n(\eps)}\in\l\{v_n\r\}$ such that $d\l(z,v_{n(\eps)}\r)<\eps.$ Because the sequence $d\l(x_i(\omega),v_{n(\eps)}\r)$ converges, there exists $k\in\nat$ such that for each $i,j\ge k$ we have
\begin{equation*} \l| d\l(x_i(\omega),v_{n(\eps)}\r) - d\l(x_j(\omega),v_{n(\eps)}\r) \r|<\eps .\end{equation*}
Therefore,
\begin{equation*} \l| d\l(x_i(\omega),z\r) - d\l(x_j(\omega),z\r) \r|<\l| d\l(x_i(\omega),v_{n(\eps)}\r) - d\l(x_j(\omega),v_{n(\eps)}\r) \r|+ d\l(z,v_{n(\eps)}\r) + d\l(z,v_{n(\eps)}\r)<3\eps,\end{equation*}
for each $i,j\ge k.$ Hence, the sequence $d\l(z,x_i(\omega)\r)$ converges and consequently also $x_i(\omega)\to x(\omega).$

As the measurability of $x$ is obvious, the proof is complete.
\end{proof}

%%%%%%%%%%%%%%%%%%%%%%%%%%%%%%%%%%%%%%%%%%%%%%%%%%%%%%%%%%%%%%%%%%%%%%%%%%%%%%%%%%%%%%%%%%%%%%%%%%%%%%%%%%%%%%%%%%%%%%%%%%%%%%%%%%%%%%%%%%%%%%%%%%%%%%%%%%%%%%%%%%%%%%%%%%%%%%%%%%%%%%%%%%%%%%%%%%%%%%%%%%%%%%%%%%
\section{Applications}

Here we apply our theorem to a few classical optimization problems (Examples \ref{exa:e1},\ref{exa:e2},\ref{exa:e3}) as well as a recent statistical model for phylogenetic inference (Example \ref{exa:bhv}). In all these examples, the resolvents of the marginal functions have a simple form and can thus be evaluated exactly. Since implicit methods are preferable to explicit ones for their better stability \cite{bertsekas}, one can conclude that Theorem \ref{thm:conv} provides us with a more powerful tool than the stochastic gradient method. In Example \ref{exa:bhv} the underlying space is a CAT(0) cubical complex without a differentiable structure, which means that minimization methods based on (sub)gradients are not applicable at all. The minimization problem in Example \ref{exa:e3} can be solved by the stochastic PPA thanks to Theorem \ref{thm:conv}, the convergence theorems by other authors mentioned in the Introduction do not apply for their too restrictive growth conditions.

In the following examples we identify a Hadamard space~$\hs,$ probability space~$S$ and random variable~$\xi$ in order to make a connection with the previous sections.
\begin{exa}[Medians] \label{exa:e1}
Let $b\col\Omega\to\rls^d$ be a random variable and set $\xi\as b$ along with $f(x,\xi)\as \l\|x-b\r\|.$ Hence $S=\rls^d.$ We are to minimize the function
\begin{equation*}
 F(x)\as\expe\l\|x-b\r\|,\qquad x\in\rls^d.
\end{equation*}
One can easily verify that the Assumptions of Theorem \ref{thm:conv} are satisfied and the resolvent of $f(\cdot,\xi)$ is easy to compute. A minimizer of $F$ is called a median.
\end{exa}
\begin{exa}[Least non-squares] \label{exa:e2}
Let $a\col\Omega\to\rls^d$ and $b\col\Omega\to\rls$ be random variables and set $\xi\as(a,b)$ along with $f(x,\xi)\as \l|\langle a,x\rangle-b\r|.$ Hence $S=\rls^{d+1}.$ Here the objective function is
\begin{equation*}
 F(x)\as\expe\l|\langle a,x\rangle-b\r|,\qquad x\in\rls^d.
\end{equation*}
The growth condition \eqref{i:sppa:lips} is satisfied provided $a\in L^2.$

Again, the resolvents can be expressed in a closed form. If $a=0,$ then $J_\lam^\xi x=x$ for every $x\in\rls^d.$ Otherwise,
\begin{equation}
 J_\lam^\xi x = \l\{
\begin{array}{ll}x-\min\l\{\lam,\frac{\langle a,x\rangle -b}{\|a\|^2} \r\}a,& \text{if} \langle a,x\rangle \ge b, \\ \\  x+\min\l\{\lam,\frac{b-\langle a,x\rangle}{\|a\|^2} \r\}a, & \text{if} \langle a,x\rangle < b, \end{array} \r.
\end{equation}
for every $x\in\rls^d.$ See \cite[Lemma 3.1]{bergmann-etal} for an explicit calculation.
\end{exa}

\begin{exa}[Least squares] \label{exa:e3}
Let $a\col\Omega\to\rls^d$ and $b\col\Omega\to\rls$ be random variables and set $\xi\as(a,b)$ along with $f(x,\xi)\as \half\l(\langle a,x\rangle-b\r)^2.$ Hence $S=\rls^{d+1}.$ We obtain the function
\begin{equation} \label{eq:leastsquares}
 F(x)\as\half\expe\l(\langle a,x\rangle-b\r)^2,\qquad x\in\rls^d.
\end{equation}
It is straightforward to show that if $a^2,ba\in L^2,$ then the condition \eqref{i:sppa:lips} is satisfied and one can also check that
\begin{equation}
 J_\lam^\xi x = x - \lam\frac{\langle a,x\rangle - b}{1+\lam \l\|a\r\|^2}a
\end{equation}
for every $x\in\rls^d;$ see \cite[Lemma 3.4]{bergmann-etal}. 

In particular, this applies into the classical least squares: given a matrix $A\in\rls^{d\times d}$ and vector $b\in\rls^d,$ minimize $F\col x\mapsto\half\|Ax-b\|^2.$ If we denote the $k$-th row of $A$ by $a_k$ and the $k$-th entry of $b$ by $b_k,$ then the loss function is 
\begin{equation} 
F(x)\as \half\|Ax-b\|^2=\frac{1}{2d}\sum_{k=1}^d \l(\langle a_k,x\rangle-b_k\r)^2, 
\end{equation}
and hence $\xi=k\in\{1,\dots,d\}$ with $f_k=\half\l(\langle a_k,\cdot\rangle-b_k\r)^2.$

One can also use a regularization of the objective function $F,$ that is, to minimize a new function $F+\mu\|\cdot\|^2,$ where $\mu>0,$ and Theorem~\ref{thm:conv} still applies and the resolvents are equally easy to compute. Recall that regularizations are used in statistics and machine learning as a standard tool against overfitting the data under consideration, as well as in optimization to handle ill-posed problems.
\end{exa}

\begin{exa}[Posterior median and mean in tree space] \label{exa:bhv}
Let $\ts_n$ be the BHV tree space whose orthant dimension is $n-2.$ This space was constructed and proved to be a CAT(0) cubical complex (hence a locally compact Hadamard space) in~\cite{bhv}. Further details and the original phylogenetic motivation can be found either in the original paper~\cite{bhv} or in~\cite{mybook}.

A recent statistical model for phylogenetic inference \cite{benner} relies upon minimizing the function
\begin{equation*}
 F(x)\as\int_{\ts_n} d(x,t)^q\di\mu_D(t),\qquad x\in\ts_n,
\end{equation*}
where $\mu_D$ stands for a posterior distribution given data $D$ and $q\in\{1,2\}.$ In~\cite{benner} the above function $F$ was approximated by empirical averages and they were minimized. Thanks to Theorem \ref{thm:conv} however, one can now minimize the function $F$ directly, since it is possible generate samples from $\mu_D$ on-the-fly. Again, the resolvents of the marginal functions are easy to compute in a closed form \cite{mm,mybook,benner}. To put this example into the perspective of Theorem \ref{thm:conv}, note that $S=\ts_n$ and $\mu=\mu_D.$ For recent developments in statistics in the BHV tree space, the interested reader is referred also to 
\cite{barden-le-owen,miller-owen-provan,nye}.
\end{exa}

%%%%%%%%%%%%%%%%%%%%%%%%%%%%%%%%%%%%%%%%%%%%%%%%%%%%%%%%%%%%%%%%%%%%%%%%%%%%%%%%%%%%%%%%%%%%%%%%%%%%%%%%%%%%%%%%%%%%%%%%%%%%%%%%%%%%%%%%%%%%%%%%%%%%%
%%%%%%%%%%%%%%%%%%%%%%%%%%%%%%%%%%%%%%%%%%%%%%%%%%%%%%%%%%%%%%%%%%%%%%%%%%%%%%%%%%%%%%
%%%%%%%%%%%%%%%%%%%%%%%%%%%%%%%%%%%%%%%%%%%%%%%%%%%%%%%%%%%%%%%%%%%%%%%%%%%%%%%%%%%%%%%%%%%%%%%%%%%%%%%%%%%%%%%%%%%%%%%%%%%%%%%%%%%%%%%%%%%%%%%%%%%%%

\bibliographystyle{siam}
\bibliography{sppa}

\end{document}